\documentclass{birkmult}

\usepackage{amssymb}
\usepackage{graphicx}

\DeclareMathAlphabet{\mathbbold}{U}{bbold}{m}{n}

\def\k{\mathbbold{k}}

\DeclareSymbolFont{rsfscript}{OMS}{rsfs}{m}{n}
\DeclareSymbolFontAlphabet{\mathrsfs}{rsfscript}

\DeclareFontFamily{OMS}{rsfs}{\skewchar\font'177}
\DeclareFontShape{OMS}{rsfs}{m}{n}{%
      <5> rsfs5
      <6> <7> rsfs7
      <8> <9> <10> rsfs10
      <10.95> <12> <14.4> <17.28> <20.74> <24.88> rsfs10
      }{}

\def\calA{\mathrsfs{A}}
\def\calB{\mathrsfs{B}}

\def\calE{\mathrsfs{E}}
\def\calF{\mathrsfs{F}}
\def\calG{\mathrsfs{G}}

\def\calO{\mathrsfs{O}}
\def\calP{\mathrsfs{P}}

\def\calR{\mathrsfs{R}}

\def\calV{\mathrsfs{V}}
\def\calW{\mathrsfs{W}}

\DeclareMathOperator{\Hom}{Hom}
\DeclareMathOperator{\Det}{Det}
\DeclareMathOperator{\gr}{gr}

\DeclareMathOperator{\As}{As}

\DeclareMathOperator{\End}{End}
\DeclareMathOperator{\Vertices}{Vertices}
\DeclareMathOperator{\Int}{Internal}
\DeclareMathOperator{\RT}{RT}

\makeatletter

\theoremstyle{plain}

\newtheorem {theorem}{Theorem}
\newtheorem {lemma}{Lemma}[theorem]
\newtheorem {corollary}{Corollary}
\newtheorem {conjecture}{Conjecture}
\newtheorem {proposition}{Proposition}

\theoremstyle{definition}

\newtheorem {definition}{Definition}
\newtheorem {remark}{Remark}
\newtheorem {example}{Example}

\let\@newpf\proof \let\proof\relax 
\newenvironment{proof}{\@newpf[\proofname]}{\qed\endtrivlist}

\makeatletter
\def\hm#1{#1\nobreak\discretionary{}{\hbox{\m@th$#1$}}{}}
\makeatother

\begin{document}

\title{Compatible associative products and trees}
\author{Vladimir Dotsenko}
\address{%
Dublin Institute for Advanced Studies,\\
10 Burlington Road,\\
Dublin 4,\\
Ireland}
\email{vdots@maths.tcd.ie}
\thanks{The author was partially supported by CNRS--RFBR grant no.~07-01-92214, by the President of the Russian Federation grant no.~NSh-3472.2008.2 and by an IRCSET research fellowship.}

\begin{abstract}
We compute dimensions and characters of the components of the operad of two compatible associative products, and give an explicit combinatorial construction of the corresponding free algebras in terms of planar rooted trees. 
\end{abstract}
\maketitle

\section{Introduction}

\subsection{Description of results.}
An algebra with two compatible associative products is a vector space $V$ equipped with two binary operations such that each of them is an associative product, and these two products are compatible (i.e., any linear combination of these products is again an associative product). Such algebras were recently studied by Odesskii and Sokolov; in~\cite{OS}, they classified simple finite-dimensional algebras of this type. In this paper, we study another extreme case: free algebras of this type. Namely, we compute dimensions of graded components of this algebra, and also give an interpretation of operations in terms of combinatorics of planar rooted trees.  

Just as for an arbitrary algebraic structure, to get information about free algebras, one first computes the $S_n$-module structure (with respect to the action by permutations of the generators) on the ``multilinear part''  (i.e., the space of elements in which each of the generators occurs exactly once) of the free algebra with $n$ generators. Then this information is used in a rather straightforward way to compute the dimensions of all graded components. 

As our computation shows, the free $1$-generated algebra with two compatible products has Catalan numbers as dimensions of its graded components. We provide a materialisation of this formula, describing two compatible products on  planar rooted trees and proving that the algebra of planar rooted trees is a free algebra with two compatible products. (We actually give a more general construction which is valid for any number of generators.) We use this construction to obtain yet another proof of the results on the Grossman--Larson algebra of planar rooted trees.

\subsection{Machinery.}
To compute dimensions and characters for spaces of multilinear elements, we use Koszul duality for operads and the theory of Koszul operads developed by V.\,Ginzburg and M.\,Kapranov. It turns out that the Koszul dual to the operad of two compatible products is much simpler than the original operad. For any Koszul operad, information on the dimensions of its components can be used to obtain similar information for the dual operad. Namely, the following assertion is true.
\begin{proposition}[\cite{GK}]Let
$f_\calO(x):=\sum_{n=1}^\infty\frac{\dim\calO(n)}{n!}x^n$.
If $\calO$ is a Koszul operad, then
 $$
f_{\calO}(-f_{\calO^!}(-x))=x.
 $$
\end{proposition}
A similar functional equation holds for the generating functions of characters of representations of the symmetric groups in the components of an operad. 

\begin{example}
For the associative operad $\As$, we have $\dim\As(n)=n!$, and so $f_{\As}(x)=\frac{x}{1-x}$. This operad is Koszul and self-dual, which agrees with the functional equation:
 $$
\frac{\frac{x}{1+x}}{1-\frac{x}{1+x}}=x. 
 $$
\end{example}

Koszulness of the operad of two compatible products was proved by Stroh\-mayer~\cite{Stro}, whose result is crucial for this paper. In our study of free algebras over $\As^2$, we also use an elegant result of Chapoton which can be used in many cases to prove that some class of algebras consists of free algebras.

\subsection{Outline of the paper.}
Throughout the paper, we assume that the reader is familiar with the main notions of operad theory. Still we briefly remind the reader of some of them when they appear in the text.

In Section~\ref{Koszul}, we recall some standard definitions of operad theory, define the operad of two compatible brackets, and list necessary facts about Koszul duality for operads.
In Section~\ref{chi-calc}, we compute the generating functions for the characters of our operads using functional equations on these generating functions, and use them to identify the corresponding representations.
In Section~\ref{monomial}, we construct a monomial basis in the multilinear part of the free algebra with two compatible products. 
In Section~\ref{freealg}, we prove that free algebras with two compatible products are free as associative algebras.
In Section~\ref{comb-oper}, we relate compatible associative products to combinatorics of trees. It turns out that the linear span of planar rooted trees has two compatible associative products, and that the resulting algebra is a free algebra with two compatible products. We also give another proof of the result of Grossman and Larson \cite{GL} on the algebra of planar rooted trees.
We conclude with some remarks and conjectural generalisations for other compatible structures.

All vector spaces and algebras throughout this paper are defined over an arbitrary field of zero characteristic.

\subsection{Acknowledgements.}
This paper was inspired by a table-talk at the Workshop on Algebraic Structures in Geometry and Physics (Leicester, July 2008). The author is grateful to Alexander Odesskii for his questions on the operad of compatible associative products, and to Andrey Lazarev for an opportunity to attend the workshop. He also wishes to thank Alexander Frolkin who read the preliminary version of this article and corrected some misprints and flaws in typesetting and language. Special thanks are to Jean-Louis Loday for useful comments.

\section{Operads: summary}\label{Koszul}

\subsection{$\mathbb{S}$-modules and operads.}
An $\mathbb{S}$-module is a collection $\{\calV(n), n\ge1\}$ of vector spaces, where each $\calV(k)$ is an $S_k$-module. Morphisms, direct sums, tensor products and duals of such objects are defined in the most straightforward way. The category of $\mathbb{S}$-modules is denoted by~$\mathbb{S}\mathrm{-mod}$.

The module $\Det$, where $\Det(n)$ is the sign representation of~$S_n$, is an important example of an $\mathbb{S}$-module. We need the following version of the dual module: $\calV^\vee=\calV^*\otimes\Det$; this is the ordinary dual twisted by the sign representation.
In some cases, we consider differential graded  $\mathbb{S}$-modules; all preceding constructions are defined for them in a similar way. The graded analogue of $\Det$ is denoted by $\calE$; the space $\calE(n)_{1-n}$ is one-dimensional and is the sign representation of the symmetric group, while all other spaces $\calE(n)_k$ are zero.

Each $\mathbb{S}$-module $\calV$ gives rise to a functor from the category $\calF\!in$ of finite sets (with bijections as morphisms) to the category of vector spaces. Namely, for a set $I$ of cardinality $n$ let
 $$
\calV(I)=\k\Hom_{\calF\!in}([n],I)\otimes_{\k S_n}\calV(n).
 $$
(Here $[n]$ stands for the ``standard'' set $\{1,2,\ldots,n\}$.)

For $\mathbb{S}$-modules $\calV$ and $\calW$, define the composition $\calV\circ\calW$ as
 $$
(\calV\circ\calW)(n)=\bigoplus\limits_{m=1}^n\calV(m)\otimes_{\k S_m}
\left(\bigoplus\limits_{f\colon [n]\twoheadrightarrow [m]}\bigotimes\limits_{l=1}^m\calW(f^{-1}(l))\right),
 $$
where the sum is taken over all surjections $f$. This operation equips the category of $\mathbb{S}$-modules with a structure of a monoidal category. An operad is a monoid in this category. See~\cite{MSS} for a more detailed definition. To simplify the definitions, we consider in this paper only operads~$\calO$ with $\calO(1)=\k$.

Let $V$ be a vector space. By definition, the operad $\End_V$ of linear mappings is the collection $\{\End_V(n)=\Hom(V^{\otimes n},V), n\ge1\}$ of all multilinear mappings of $V$ into itself with the obvious composition maps.

Using the operad of linear mappings, we can define an algebra over an operad $\calO$;
a structure of such an algebra on a vector space is a morphism of the operad $\calO$ into the corresponding operad of linear mappings. Thus an algebra over an operad $\calO$ is a vector space $W$ together with a collection $\calO(n)\otimes_{\k S_n} W^{\otimes n}\to W$ of mappings with obvious compatibility conditions. The free algebra generated by a vector space $X$ over an operad $\calO$ is
 $$\calO(X):=\calO\circ X=\bigoplus_{k=1}^\infty\calO(n)\otimes_{\k S_n}X^{\otimes n}.$$

\subsection{Operads defined by generators and relations.}
The free operad $\calF_\calG$ generated by an $\mathbb{S}$-module $\calG$ (with $\calG(1)=0$) is defined as follows. A basis in this operad consists of some species of trees. These trees have a distinguished root (of degree one). A tree belonging to $\calF_\calG(n)$ has exactly $n$ leaves, its internal vertices (neither leaves nor the root) labelled by basis elements of $\calG$, any vertex with $k$ siblings being labelled by an element of $\calG(k)$.
The unique tree whose set of internal vertices is empty generates the one-dimensional space  $\calF_\calG(1)$. The composition of a tree $t$ with $l$ leaves and trees $t_1,\ldots,t_l$ glues the roots of $t_1,\ldots,t_l$ to the respective leaves of $t$. (In every case, two edges glued together become one edge, and the common vertex becomes an interior point of this edge.)

Free operads are used to define operads by generators and relations. Let $\calG$ be an $\mathbb{S}$-module, and let $\calR$ be an $\mathbb{S}$-submodule in $\calF_\calG$.
An (operadic) ideal generated by $\calR$ in $\calF_\calG$ is the linear span of all trees such that at least one internal vertex is labelled by an element of~$\calR$.
An operad with generators $\calG$ and relations~$\calR$ is the quotient of the free operad~$\calF_\calG$ modulo this ideal.

\begin{definition}
The associative operad $\As$ is generated by one binary operation $\star\colon a,b\mapsto a\star b$. The relations in this operad are equivalent to the associativity condition for every algebra over this operad:
 $$
(a\star b)\star c=a\star (b\star c).
 $$

The operad of two compatible associative products $\As^{2}$ is generated by two binary operations (products) $\star_1$ and $\star_2$. The relations in this operad are equivalent to the following identities in each algebra over this operad: the associativity conditions
\begin{gather*}
(a\star_1 b)\star_1 c=a\star_1 (b\star_1 c),\\
(a\star_2 b)\star_2 c=a\star_2 (b\star_2 c)
\end{gather*}
for products and the four-term relation
\begin{equation}\label{compat}
(a\star_1 b)\star_2 c+(a\star_2 b)\star_1 c=a\star_1 (b\star_2 c)+a\star_2 (b\star_1 c)
\end{equation}
between the products.
\end{definition}

\subsection{Koszul duality for operads.}

Let an operad $\calO$ be defined by a set of binary operations $\calB$ with 
quadratic relations $\calR$ (that is, relations involving ternary operations obtained by compositions from the given binary operations). In this case, $\calO$ is said to be quadratic. For quadratic operads, there is an analogue of Koszul duality for quadratic algebras. To a quadratic operad $\calO$, this duality assigns the operad $\calO^!$ with generators $\calB^\vee$ and with the annihilator of $\calR$ under the natural pairing as the space of relations. Just as in the case of quadratic algebras, $(\calO^!)^!\simeq\calO$. 

\begin{example}[\cite{GK}]
The operad $\As$ is self-dual: $\As^!\simeq\As$.
\end{example}

The cobar complex $\mathbf{C}(\calO)$ of an operad $\calO$ is the free operad with generators $\{\calO^*(n), n\ge2\}$ equipped with a differential $d$ with $d^2=0$ (see \cite{MSS} for details). Once again we use twisting by the sign, now to get another version of the cobar complex, $\mathbf{D}(\calO)=\mathbf{C}(\calO)\otimes\calE$.
The zeroth cohomology of~$\mathbf{D}(\calO)$ is isomorphic to the operad~$\calO^!$.

\begin{definition}
An operad $\calO$ is said to be Koszul if $H^i(\mathbf{D}(\calO))=0$ for~$i\ne0$.
\end{definition}

\begin{proposition}[\cite{Stro}]
The operad $\As^{2}$ is Koszul.
\end{proposition}

\subsection{Generating functions and characters.}
As we mentioned above, to each operad (and more generally, to each $\mathbb{S}$-module) $\calO$ one can assign the formal power series (the exponential generating function of the dimensions) 
 $$
f_\calO(x)=\sum_{n=1}^\infty\frac{\dim\calO(n)}{n!}x^n,
 $$
and if $\calO$ is a Koszul operad, then
 $$
f_{\calO}(-f_{\calO^!}(-x))=x.
 $$
This functional equation is an immediate corollary of a functional equation relating more general generating functions that will be defined now.

The character of a representation $M$ of the symmetric group $S_n$ can be identified \cite{Mac}
with a symmetric polynomial $F_M(x_1,x_2,\ldots)$ of degree~$n$ in infinitely many variables.
To each $\mathbb{S}$-module $\calV$ we assign the element 
 $$
F_\calV(x_1,\ldots,x_k,\ldots)=\sum_{n\ge1}F_{\calV(n)}(x_1,\ldots,x_k,\ldots)
 $$ 
of the algebra $\Lambda$ of symmetric functions. This algebra is the completion of the algebra of symmetric polynomials in infinitely many variables with respect to the valuation defined by the degree of a polynomial. It is isomorphic to the algebra of formal power series in Newton power sums $p_1,\ldots,p_n,\ldots$. The series $F_\calV$ is a generating series for the characters of symmetric groups. Namely, by multiplying the coefficient of $p_1^{n_1}\ldots p_k^{n_k}$ by
$1^{n_1}n_1!\ldots k^{n_k}n_k!$, we obtain the value of the character of $\calV(n)$ on a permutation whose decomposition into disjoint cycles contains $n_1$ cycles of length~$1$, \ldots, $n_k$ cycles of length~$k$.
This definition can be generalised to the case of differential graded modules; for such a module $\calV=\bigoplus_i\calV_i$ we set $F_\calV=\sum_i(-1)^iF_{\calV_i}$ (the Euler characteristic of $\calV$).

If $\calV$ is equipped with an action of a group $G$ commuting with the action of the symmetric groups, then for each $n$ the group $S_n\times G$ acts on the space $\calV(n)$. In this case, to $\calV$ we assign an element of the algebra $\Lambda_{G}$ of symmetric functions over the character ring of $G$ (or, in other words, a character of $G$ ranging over symmetric functions). We denote this element by $F_\calV(x_1,\ldots,x_n,\ldots;g)$, where $g\in G$. 

\begin{remark}\label{specialisations}
Further in this text we use the following properties of specialisations of our generating functions:
\begin{enumerate}
 \item $f_\calV(x)=F_\calV(p_1,\ldots)|_{p_1=x,\,p_2=p_3=\ldots=0}$;
 \item $F_\calV(p_1,\ldots)=F_\calV(p_1,\ldots;g)|_{g=e}$, where $e$ is the identity element of the group $G$;
 \item for any finite-dimensional vector space $V$ (considered as an $\mathbb{S}$-module concentrated in degree~$1$),
 $$
f_{\calV(V)}(x)=F_\calV(p_1,\ldots)|_{p_1=x\dim V,\,p_2=x^2\dim V,\,p_3=x^3\dim V,\ldots},
 $$
where $\calV(V)=\calV\circ V$, and $f_{\calV(V)}(x)$ is the generating function for dimensions of the (graded) vector space $\calV(V)$.
\end{enumerate}
\end{remark}

\subsection{Functional equation for characters.}
\begin{definition}
Fix $H(x_1,x_2,\ldots;g)\in\Lambda_G$. The plethysm corresponding to $H$ (the plethystic substitution of~$H$) is the algebra homomorphism $F\mapsto F\circ H$ of $\Lambda_G$ into itself that is linear over the character ring of $G$ and is defined on symmetric functions by
$p_n\circ H=H(x_1^n,x_2^n,\ldots; g^n)$.
\end{definition}

In particular,
 $$
p_n\circ (H(p_1,p_2,\ldots,p_k,\ldots;g))=H(p_n,p_{2n},\ldots,p_{kn},\ldots;g^n).
 $$

Let $\varepsilon$ be the involution of $\Lambda_G$ linear over the character ring of $G$ and taking $p_n$ to $-p_n$. 

\begin{theorem}[\cite{DK}]
Suppose that the operad $\calO$ is Koszul. Then the following equation holds in $\Lambda_G$: 
 $$
\varepsilon(F_{\calO})\circ\varepsilon(F_{\calO^!})=p_1.
 $$
\end{theorem}

\section{Calculation of dimensions and characters}\label{chi-calc}

Note that the components of $\As^2$ and $(\As^2)^!$ are equipped with an action of $SL_2$ (arising from the action on the space of generators of the operad~$\As^2$), which commutes with the action of the symmetric groups. All information about these operads will follow from the functional equation on the characters and the explicit description of the representation $(\As^2)^!(n)$ of the group $SL_2\times S_n$.

The character ring of $SL_2$ is isomorphic to the ring of Laurent polynomials in one variable~$q$ (for example, the character of the $n$-dimensional irreducible representation $L(n-1)$ is equal to $\frac{q^n-q^{-n}}{q-q^{-1}}$); the element of $\Lambda_{SL_2}$ corresponding to an $\mathbb{S}$-module $\calV$ is denoted by $F_\calV(p_1,\ldots,p_n,\ldots;q)$. This notation differs a little from the one introduced above, but we hope that it will not lead to a confusion. In this case, the plethysm is defined as follows: $p_n\circ q=q^n$.

\begin{theorem}
For the operad of two compatible associative products, the following holds:
\begin{gather*}
f_{\As^2}(x)=\sum_{n\ge1}c_nx^n,\\ 
F_{\As^2}(p_1,\ldots)=\sum_{n\ge1}c_np_1^n,\\
F_{\As^2}(p_1,\ldots;q)=\sum_{n\ge1}p_1^n(q^{n-1}+N_{n,1}q^{n-3}+\ldots+N_{n,k}q^{n-1-2k}+\ldots),
\end{gather*}
where $$c_n=\frac{1}{n+1}\binom{2n}{n}$$ are Catalan numbers, and $$N_{n,k}=\frac{1}{n}\binom{n}{k}\binom{n}{k+1}$$ are Narayana numbers.
\end{theorem}

\begin{proof}
Note that the substitution $p_1=t$ transforms the second formula into the first one, and the substitution $q=1$ transforms the third equation into the second one (Narayana numbers refine Catalan numbers \cite{Stan2}). Thus, the third statement implies the other two, so we shall restrict ourselves to proving only the former. From the results of Strohmayer \cite{Stro} (combined with results on $SL_2$-modules from \cite{DK}), it follows that as an $S_n\times SL_2$-module,  
 $$
(\As^2)^!(n)=\mathbb{Q}S_n\otimes L(n-1).
 $$ 
Thus, the $S_n\times SL_2$-character of the Koszul dual operad is given by the formula
 $$
\sum_{n\ge1}\frac{q^n-q^{-n}}{q-q^{-1}}p_1^n=\frac{p_1}{(1-qp_1)(1-q^{-1}p_1)}. 
 $$ 
The functional equation for characters implies that the character $$F_{\As^2}(p_1,\ldots,p_n,\ldots;q)$$ satisfies 
\begin{equation}\label{FuncAs}
\frac{F_{\As^2}}{(1+qF_{\As^2})(1+q^{-1}F_{\As^2})}=p_1.
\end{equation}
From this equation it is clear that $F_{\As^2}(p_1,\ldots;q)$ depends only on $p_1$ (and $q$).

On the other hand, it is well known (see, for example, \cite{Stan2}, which should be adjusted to our parametrisation of Narayana numbers) that the generating function
 $$
N(t,x)=1+\sum_{n=1}^{\infty}\sum_{k=0}^{n-1}N_{n,k}t^nx^k
 $$
of Narayana numbers satisfies the equation
\begin{equation}\label{Nara}
txN^2(t,x)-txN(t,x)+tN(t,x)-N(t,x)+1=0.
\end{equation}
It is easy to see that the third statement of the theorem is equivalent to the following equation for generating functions:
 $$
N(p_1q^{-1},q^2)=1+q^{-1}F_{\As^2}(p_1,q).
 $$
Let
 $$
N(p_1q^{-1},q^2)=1+q^{-1}G(p_1,q).
 $$
From the functional equation \eqref{Nara} we deduce that $G$ satisfies
 $$
p_1q(1+q^{-1}G)^2-p_1q(1+q^{-1}G)+\frac{p_1}{q}(1+q^{-1}G)-(1+q^{-1}G)+1=0,
 $$
which can be rewritten as
 $$
p_1(1+(q+q^{-1})G+G^2)=G.
 $$
The latter equation coincides with equation \eqref{FuncAs}, and determines $G$ uniquely, thus $G=F_{\As^2}$
\end{proof}

\begin{corollary}
\begin{enumerate}
 \item As $S_n$-module, $\As^2(n)$ is free of rank $c_n$.
 \item As $S_n\times SL_2$-module, 
 $$
\As^2(n)\simeq\mathbb{Q}S_n\otimes(L(n-1)+L(n-3)^{N_{n,1}-1}+L(n-5)^{N(n,2)-N(n,1)}+\ldots).
 $$
\end{enumerate}
\end{corollary}

\begin{proof}
The $SL_2$-character of the module
 $$
L(n-1)+L(n-3)^{N_{n,1}-1}+L(n-5)^{N(n,2)-N(n,1)}+\ldots
 $$
is equal to 
 $$
q^{n-1}+N_{n,1}q^{n-3}+N_{n,2}q^{n-5}+\ldots+N_{n,k}q^{n-1-2k}+\ldots+N_{n,n-1}q^{1-n},
 $$
so the second statement follows. The first statement is obtained from the second one, if we forget about the $SL_2$-action.
\end{proof}

\begin{corollary}\label{dimfreealg}
The dimension of the $k^\text{th}$ component of the free $\As^2$-algebra generated by a vector space~$V$ is equal to $c_k(\dim V)^k$. In particular, the dimension of the $k^\text{th}$ graded component of the free $\As^2$-algebra with one generator is equal to the $k^\text{th}$ Catalan number.
\end{corollary}

\begin{proof}
This follows immediately from our previous results: we just apply the third formula of Remark~\ref{specialisations} to $\As^2(V)$ (the free $\As^2$-algebra generated by~$V$).
\end{proof}

\section{A monomial basis for $\As^2$}\label{monomial}

In this section we describe a monomial basis for $\As^2$. One can compare the methods and structure of this paragraph to the same in \cite{DK} in the case of the operad of two compatible Lie brackets.

\begin{definition}
Given a finite ordered set $$A\hm=\{a_1,a_2,\ldots,a_n\},$$ $a_1<a_2<\ldots<a_n$, define a family of monomials $\mathfrak{B}(A)$ in the free $\As^2$-algebra generated by $A$ recursively. Our recursive definition also assigns to a monomial $m$ its ``top level operation'' $t(m)\in\{1,2\}$, which is used to define further monomials.
\begin{itemize}
\item For $A=\{a_1\}$, let $\mathfrak{B}(A)=\{a_1\}$, and let $t(a_1)=1$.
\item For $n>1$, a monomial $b$ belongs to $\mathfrak{B}(A)$ if and only if it satisfies one of the two conditions:
\begin{enumerate}
\item $b=a_k\star_1 b'$, where $1\le k\le n$ and $b'\in\mathfrak{B}(A\setminus\{a_k\})$; in this case we put $t(b)=1$.
\item $b=b_1\star_2 b_2$, where $b_1\in\mathfrak{B}(A_1)$, $b_2\in\mathfrak{B}(A_2)$
for some $A_1 \sqcup A_2\hm=A$, and $t(b_1)=1$; in this case we put $t(b)=2$.
\end{enumerate}
\end{itemize}
\end{definition}

\begin{theorem}
The family of monomials $\mathfrak{B}(A)$ provides a basis for $\As^2(A)$. 
\end{theorem}

\begin{proof}
We shall prove that this family spans the component $\As^2(A)$, and that its number of elements is equal to the dimension of this component. It will follow that it has to be a basis. 

\begin{lemma}
The family of monomials $\mathfrak{B}(A)$ spans $\As^2(A)$.
\end{lemma}

\begin{proof}
Consider a monomial $m$. It is a product of two monomials, and by induction we can assume that they both belong to families $\mathfrak{B}(A')$ for some sets $A'\subset A$. Using the associativity property for each of the products, we are left with only one case in which $m$ does not belong to $\mathfrak{B}(A)$, namely
 $$
m=(m_1\star_2 m_2)\star_1 m_3
 $$ 
for some $m_1$, $m_2$, $m_3$. In this case, we use the compatibility relation \ref{compat}:
 $$
m=m_1\star_2(m_2\star_1 m_3)+m_1\star_1(m_2\star_2 m_3)-(m_1\star_1 m_2)\star_2 m_3,
 $$
which shows that we can proceed by induction: in the first two summands, the degree of the first factor has decreased, and the last summand has fewer products of the second type in its first factor.
\end{proof}

\begin{lemma}
The number of elements in $\mathfrak{B}(A)$ is equal to $\frac{(2|A|)!}{(|A|+1)!}$. 
\end{lemma}

\begin{proof}
Let $\beta_n=|\mathfrak{B}([n])|$. Moreover, for $i=1,2$ let $\beta_{i,n}=|\mathfrak{B}_i([n])|$, where $\mathfrak{B}_i([n])$ is the set of all monomials $b\in\mathfrak{B}([n])$ with $t(b)=i$. We use exponential generating functions again:
 $$
\beta(x)\hm=\sum_{l\ge1}\frac{\beta_lx^l}{l!}, \quad \beta_{i}(x)\hm=\sum_{l\ge1}\frac{\beta_{i,l}x^l}{l!},
 $$
The first condition implies that
 $$
\beta_{1,n+1}=(n+1)\beta_n,
 $$
which can be rewritten as 
\begin{equation}\label{eqbeta}
\beta_1(x)-x=x\beta(x).
\end{equation}
The definition of $\mathfrak{B}([n])$ basically means that on the level of $\mathbb{S}$-modules,
 $$
\mathfrak{B}_2(n)=(\As\circ\mathfrak{B}_1)(n)
 $$ 
for $n\ge 2$, so
 $$
\beta_2(x)=(f_{\As}(x)-x)\circ\beta_1(x).
 $$
Let us rewrite this equation using the formulae 
 $$\beta_1(x)+\beta_2(x)=\beta(x)$$ and $$f_{\As}(x)=\frac{x}{1-x}.$$ 
We get 
\begin{equation}\label{as-circ-beta1}
\beta(x)=\frac{\beta_1(x)}{1-\beta_1(x)}.
\end{equation}
This formula can be written in the form $\beta_1(x)=\frac{\beta(x)}{1+\beta(x)}$. Now we can substitute it into \eqref{eqbeta}, and get the equation
 $$
\frac{\beta(x)}{1+\beta(x)}=x(1+\beta(x)),
 $$
which coincides with the functional equation for $f_{\As^2}(x)$ obtained from \eqref{FuncAs} by setting $q=1$.
\end{proof}
\end{proof}

\section{Free algebras over $\As^2$}\label{freealg}

In this section, we prove that any free algebra with two compatible products is free as an associative algebra. Let us recall a theorem due to Chapoton \cite{Chap} that is one of the main ingredients in our proof.

\subsection{Chapoton's theorem}

Let $\calP$ be an operad. Assume that $\calP(1)=\mathbb{Q}$, and let $\calP^+$ be the
$\mathbb{S}$-module such that $\calP=\calP(1)\oplus\calP^+$. Let $\calA$ be a $\calP$-algebra
in the category of $\mathbb{S}$-modules. The structure of a $\calP$-algebra on $\calA$ 
is given by a morphism $\mu\colon\calP\circ\calA\to\calA$.

Let us define a decreasing $\calP$-algebra filtration of $\calA$: for each $k\ge 0$ we a define a subspace $\calA_{\geq k}$ of $\calA$. Let
$\calA_{\geq 0}$ be $\calA$, and for $k>0$ let $\calA_{\geq k}$ be the image under $\mu$ of $\calP^+ \circ\calA_{\geq k-1}$.

We will assume that this filtration is separating, which is true, for instance, if $\calA$ has a grading concentrated in positive degrees.

Let us define $H_0(\calA)$ to be the degree $0$ component $\calA_{\ge 0}/\calA_{\ge 1}$ of the associated graded $\calP$-algebra $\gr\calA$.

Let us choose a section of $H_0(\calA)$ in $\calA$. Consider $\calP(H_0(\calA))$, that is the free $\calP$-algebra generated by $H_0(\calA)$. Then there exists a unique morphism $\theta$ of $\calP$-algebras from $\calP(H_0(\calA))$ to $\calA$ extending the chosen section.

\begin{theorem}[\cite{Chap}]
The morphism $\theta$ is surjective. Therefore, if dimensions (or graded characters) of $\calA$ and $\calP(H_0(\calA))$ are equal, then $\theta$ is an isomorphism.
\end{theorem}

\subsection{Free algebras with two compatible products are free}

\begin{theorem}\label{freearefree}
Free algebras with two compatible brackets are free as associative algebras. 
\end{theorem}

\begin{proof}
Let us first prove that there exists an $\mathbb{S}$-module $\calG$ such that the $\mathbb{S}$-modules $\As^2$ and $\As\circ\calG$ are isomorphic. To do that, we apply Chapoton's result in the case $\calP=\As$ and $\calA=\As^2$, where the $\calP$-algebra structure is given by the second product. Thus, Chapoton's theorem suggests that we should put $\calG:=H_0(\calA)$, so that in order to prove our theorem, we only need to prove that graded characters of $\As^2$ and $\As\circ\calG$ are equal. This is guaranteed by the next lemma.

\begin{lemma}
\begin{enumerate}
 \item For each component of $\As^2$, the part of its monomial basis consisting of elements $b$ with $t(b)=1$ can be taken as a lift $\theta\colon\calG\to\As^2$; 
 \item $f_{\As^2}(x)=f_{\As}\circ f_{\calG}(x)$.
\end{enumerate}
\end{lemma}
 
\begin{proof}
From our proof of the spanning property, it follows that any monomial for which the top level operation is the second product belongs to the subspace spanned by all basis elements $b$ with $t(b)=2$, so the quotient by the space of all such monomials is identified with the complementary subspace.

Also, the equation \eqref{as-circ-beta1} is
 $$
f_{\As^2}(x)=f_{\As}\circ\beta_1(x),
 $$
which is exactly what our second statement claims.
\end{proof}

Now we are ready to prove our theorem. For a vector space $V$,
 $$
\As^2(V)\simeq\As(\calG(V)),
 $$ 
so the free $\As^2$-algebra with generators $V$ is isomorphic to the free associative algebra with generators $\calG(V)$.
\end{proof}

\section{Labeled rooted trees and compatible products}\label{comb-oper}

A planar rooted tree is an abstract rooted tree with a linear order on the set of children of every vertex. Alternatively, one can imagine a tree embedded into the plane in such a way that all children of any vertex~$v$ have their $y$-coordinates less than the $y$-coordinate of~$v$ (in this case, the linear order appears from reading the outgoing edges from left to right). 

\begin{proposition}[\cite{Stan2}]
The number of planar rooted trees with $n+1$ vertices is equal to the Catalan number $c_n$. 
\end{proposition}

Thus, if we consider the planar rooted trees with $k+1$ vertices equipped with a labelling of all non-root vertices  by elements of some finite set $S$, the number of these objects is equal to $c_k(\# S)^k$, which is, by Corollary~\ref{dimfreealg}, equal to the dimension of the $k^\text{th}$ component of the free $\As^2$-algebra generated by~$S$. In the remaining part of this section, we show that this fact is not a mere coincidence. Namely, we define two compatible associative products on the linear span of all planar rooted trees with $S$-labelled non-root vertices, and show that this linear span is free as an $\As^2$-algebra. 

Denote by $\RT(S)$ the collection of all planar rooted trees whose non-root vertices are labelled by elements of a finite set $S$ (possibly with repeated labels). We start by defining several operations on the linear span $\mathbb{Q}\RT(S)$. 

\begin{definition}
Let $T_1,T_2\in\RT(S)$.
Define the tree $T_1\cdot T_2$ as the tree obtained by identifying the roots of $T_1$ and $T_2$; the linear ordering of the children of this vertex is uniquely defined by the condition that all children coming from $T_1$ preceed all children coming from $T_2$. This operation is associative, and every $T\in\RT(S)$ can be uniquely decomposed as $T=T[1]\cdot T[2]\cdot\ldots\cdot T[k]$, where for each tree $T(j)$ its root has only one child.
\end{definition}

Let us denote by $\Vertices(T)$ the set of all vertices of a tree~$T\in\RT(S)$ (including the root), and by $\Int(T)$ the set of all internal vertices of~$T$.

\begin{definition}
Let $T_1,T_2\in\RT(S)$. Assume that the root of $T_1$ has $k$ children (as in the above definition).
\begin{enumerate}
 \item  To every mapping $f\colon[k]\to\Vertices(T_2)$ we assign a new tree $T_1\circ^f T_2$ which is obtained as follows. For each $v\in\Vertices(T_2)$ we let $$f^{-1}(v)=\{i_1<\ldots<i_s\},$$ form a tree $T_1[i_1]\cdot\ldots\cdot T_1[i_s]$, and identify the root of this tree with the vertex $v$ (keeping the label of $v$) in a way that all children of this tree are placed left of all the children of $v$ in~$T_2$.
 \item  To every mapping $g\colon[k]\to\Int(T_2)$ we assign a new tree $T_1\circ_g T_2$ which is obtained as follows.
For each $v\in\Int(T_2)$ we let $$g^{-1}(v)=\{j_1<\ldots<j_r\},$$ form a tree $T_1[j_1]\cdot\ldots\cdot T_1[j_r]$, and identify the root of this tree with the vertex $v$ (keeping the label of $v$) in a way that all children of this tree are placed left of all the children of $v$ in~$T_2$.
\end{enumerate}
\end{definition}

Let us now define two products on $\mathbb{Q}\RT(S)$.

\begin{definition}
Let $T_1,T_2\in\RT(S)$. Assume that the root of $T_1$ has $k$ children. Define the products $T_1\star_1 T_2$ and $T_1\star_2 T_2$ as follows:
\begin{gather}
T_1\star_1 T_2=\sum_{f\colon[k]\to\Vertices(T_2)}T_1\circ^f T_2,\label{star1}\\
T_1\star_2 T_2=\sum_{g\colon[k]\to\Int(T_2)}T_1\circ_g T_2\label{star2}.
\end{gather}
\end{definition}

\begin{example}
For the trees 
 $$
\includegraphics{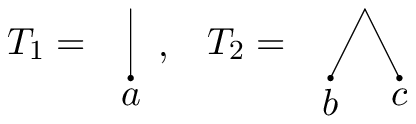}
 $$ 
the product $T_1\star_1 T_2$ is equal to
 $$
\includegraphics{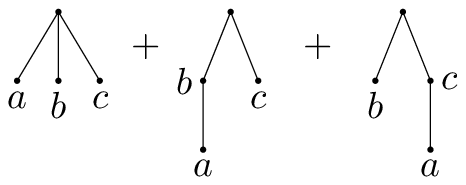}
 $$ 
while the product $T_1\star_2 T_2$ is equal to
 $$
\includegraphics{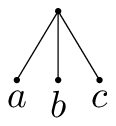}
 $$ 
 
\end{example}

\begin{theorem}\label{treefreealg}
\begin{enumerate} 
 \item The products $\star_1$ and $\star_2$ are associative and compatible with each other.
 \item The $\As^2$-algebra $\mathbb{Q}\RT(S)$ is isomorphic to the free $\As^2$-algebra generated by~$S$.
\end{enumerate}
\end{theorem}

\begin{example}
For the trees 
 $$
\includegraphics{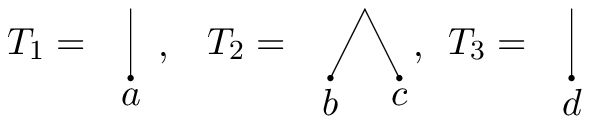}
 $$ 
the four products that occur in the compatibility relation~\ref{compat} are as follows. The product $T_1\star_1(T_2\star_2 T_3)$ is
 $$
\includegraphics[scale=0.7]{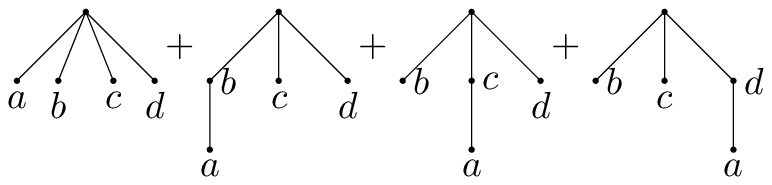}
 $$ 
the product $T_1\star_2(T_2\star_1 T_3)$ is
 $$
\includegraphics[scale=0.7]{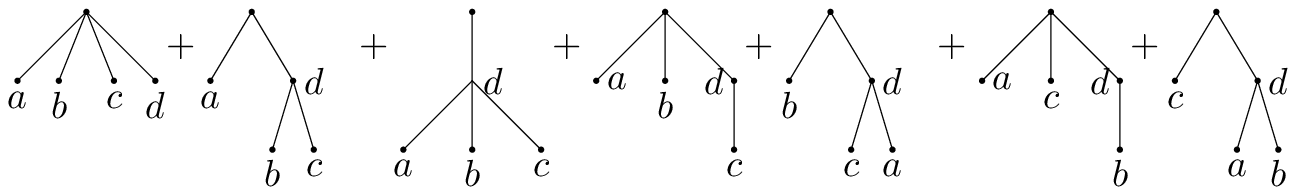}
 $$ 
the product $(T_1\star_2 T_2)\star_1 T_3$ is
 $$
\includegraphics[scale=0.7]{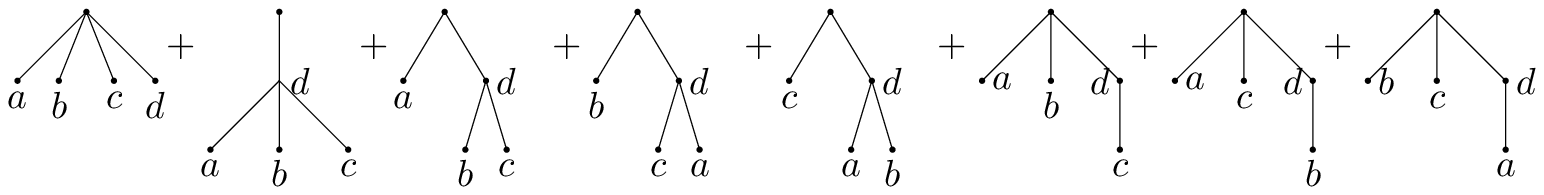}
 $$ 
and, finally, the product $(T_1\star_1 T_2)\star_2 T_3$ is
 $$
\includegraphics[scale=0.7]{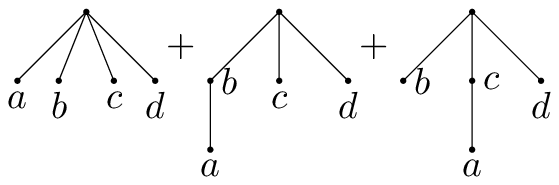}
 $$ 
and the compatibility condition is satisfied.
\end{example}

\begin{proof}
The associativity conditions for both products are pretty transparent; to show that all the terms in consecutive product $\Pi_1=T_1\star_1 (T_2\star_1 T_3)$ appear in the product $\Pi_2=(T_1\star_1 T_2)\star_1 T_3$ one should just notice that to obtain the terms in $\Pi_1$ where subtrees of $T_1$ are attached directly to vertices of $T_3$ (all other terms appear in $\Pi_2$ for tautological reasons) we should just attach the corresponding subtrees to the root of $T_2$ when computing $T_1\star_1 T_2$ for $\Pi_2$, then we can attach them as required when computing the final product. The same argument works for the second product. 

We shall establish the compatibility condition rewritten in the form 
 $$
(T_1\star_2 T_2)\star_1 T_3-T_1\star_2(T_2\star_1 T_3)=T_1\star_1 (T_2\star_2 T_3)-(T_1\star_1 T_2)\star_2 T_3.
 $$
The reason is that for our products both the left hand side and the right hand side are combinations of trees with nonnegative coefficients, and we can interpret the summands that appear there in a rather nice and simple way. Namely, the trees that appear on the left hand side are those for which there exist subtrees of $T_1$ that are attached to some leaves of $T_3$. Obviously, the left hand side has the same interpretation. Details are simple and we leave them to the reader.

Now we shall prove that the algebra $\mathbb{Q}\RT(S)$ is free as an $\As^2$-algebra. Note that this algebra admits a natural grading by the number of non-root vertices of a tree, and the dimensions of graded components are precisely the dimensions of the graded components of the free $\As^2$-algebra generated by $S$. It remains to show that our algebra is generated as an $\As^2$-algebra by elements of degree~$1$; it will follow that it is a quotient of the corresponding free algebra, and since it has the same dimensions of graded components, these two algebras should be isomorphic. Thus, it remains to prove the following lemma.
\begin{lemma}
As an $\As^2$-algebra, $\mathbb{Q}\RT(S)$ is generated by elements of degree~$1$. 
\end{lemma}
\begin{proof}
We use induction on degree. Assume that all trees of degree at most $k$ are products of elements of degree~$1$. Let us show that the same holds for trees of degree $k+1$. For the fixed degree, we use induction on the number of children of the root vertex. First of all, if the root of a tree has only one child, the label of this child is $s$, and the subtree obtained by forgetting about the root and about the label of the child is $T'$, then our tree is $T'\star_1 s-T'\star_2 s$, where $s$ denotes the tree with two vertices whose non-root vertex is labelled by~$s$, so we can proceed by induction on degree. Assume that the root has at least two children; then $T=T'\cdot T''$, where the root of $T'$ has only one child. Let us compute $T'\star_1 T''$. This product contains $T$ as a summand; for all other summands the root of the result has fewer children than~$T$ (so we can proceed by induction).
\end{proof}
\end{proof}

\begin{remark}
Theorem \ref{treefreealg} can be used to obtain an alternative proof of one of the main results of the article \cite{GL}. Namely, since the first product $T_1\star_1 T_2$ is the Grossman--Larson product on $\mathbb{Q}\RT(S)$, it follows from Theorem \ref{freearefree} that the Grossman--Larson algebra of planar rooted trees is a free associative algebra; moreover, from our proofs it is easy to see that as a generating set of this algebra we can take all trees whose root has only one child. These are exactly the results of Grossman and Larson. 
\end{remark}

\section{Appendix: remarks and open questions}\label{conj}

\subsection{Relation to Grossman--Larson Hopf algebra structure}
Recall the Grossman--Larson algebra of planar rooted trees from \cite{GL} is originally defined as a Hopf algebra, with the coproduct defined as follows. 
\begin{definition}[\cite{GL}]
Define the coproduct 
 $$
\Delta\colon\mathbb{Q}\RT(S)\to\mathbb{Q}\RT(S)\otimes\mathbb{Q}\RT(S) 
 $$
by the formula
 $$
\Delta(T)=\sum_{I\sqcup J=[k]}T[i_1]\cdot\ldots\cdot T[i_p]\otimes T[j_1]\cdot\ldots\cdot T[j_q],
 $$
where $T=T[1]\cdot\ldots\cdot T[k]\in\mathbb{Q}\RT(S)$, $I=\{i_1<\ldots<i_p\}$, $J=\{j_1<\ldots<j_q\}$. 
\end{definition}
One could ask what is the relation between this coproduct and the second product that we introduced. 

\begin{proposition}
Consider $\mathbb{Q}\RT(S)$ as an associative algebra with respect to either of the products $\star_1$, $\star_2$ (and introduce the product on its tensor square accordingly). Then $\Delta$ is an algebra homomorphism.  
\end{proposition}
\begin{proof}
For the first product, this statement is proved in~\cite{GL}. For the second product, one can use the same proof with some slight modifications (basically, what should be done is simply forgetting all summands where grafting to leaves occurs).  
\end{proof}

\begin{remark}
It is worth mentioning that although the tensor product of two $\As^2$-algebras can be turned into an $\As^2$-algebra in many different ways, two products on the tensor square of the free algebra that we just described are not compatible; the family of products 
 $$
(a_1\otimes b_1)\star_{\lambda,\mu}(b_1\otimes b_2)=
(\lambda a_1\star_1 a_2+\mu a_1\star_2 a_2)\otimes (\lambda b_1\star_1 b_2+\mu b_1\star_2 b_2)
 $$
is a pencil of associative products, but it is not a linear pencil anymore (they rather resemble pencils of associative products from~\cite{Mo}). Thus, the relationship between Hopf algebra structure and the structure of an algebra with two compatible products is yet to be clarified.
\end{remark}

\subsection{Relation to other operads realised by planar trees}

\begin{remark}[\cite{Lo_pc}]
Consider the operad $\calO_q$ generated by two binary operations $\circ$ and $\bullet$ which satisfy the relations
\begin{gather*}
(x\circ y)\circ z=x\circ(y\circ z),\\
(x\bullet y)\circ z + q (x\circ y)\bullet z = x\bullet (y\circ z) + q\ x\circ (y\bullet z),\\
(x\bullet y)\bullet z=x\bullet(y\bullet z).
\end{gather*}
Then $\calO_0$ is the operad $Dup$ of duplicial algebras~\cite{Lo}, while $\calO_1$ is the operad $\As_2$.

Also, consider the operad $\calP_t$ generated by two binary operations $\prec$ and $\succ$ which satisfy the relations
\begin{gather*}
 (x\prec y)\prec z = x\prec (y\prec z) + t\ x\prec (y\succ z),\\
 (x\succ y)\prec z = x\succ (y\prec z),\\
 (x\succ y)\succ z + t (x\prec y)\succ z= x\succ (y\succ z).
\end{gather*}
Then $\calP_0$ is the operad $Dup$, while $\calP_1$ is the operad $Dend$ of dendriform algebras~\cite{Lo}.
\end{remark}

It is known that free algebras over $Dend$ and $Dup$ can be realised by planar trees. It would be interesting to define in a pure combinatorial way a $2$-parameter family of pairs of binary operations $\star_{1,q,t}$ and $\star_{2,q,t}$ on $\mathbb{Q}\RT(S)$ which have correct specialisations to $q=t=0$ (duplicial case) $q=1$, $t=0$ (compatible associative products) and $q=0$, $t=1$ (dendriform case).

\subsection{Other Hopf-algebraic families of trees}

Some general phenomenon that we think is worth mentioning here is the existence of compatible associative products for many other well known algebras where the product is described via combinatorics of trees. The main idea is very simple. If the product in the linear span of rooted trees (planar or not) is defined for two trees 
$T_1=T_1[1]\cdot T_1[2]\cdot\ldots\cdot T_1[k]$ and $T_2$ as the sum of all graftings of some type of trees $T_1[i]$ to vertices of the tree $T_2$, then another product over all graftings of the same type but only to internal vertices is compatible with the first product. For algebras of planar \emph{binary} trees (which also often occur in literature) an analogous recipe holds: if a product is defined in terms of graftings, then graftings only to the ``left-going'' leaves produce a compatible product.

\begin{example}
Consider the Connes--Kreimer Hopf algebra of renormalisation \cite{CK}, which is a polynomial algebra on (abstract) rooted trees, or, in other words, an algebra on the linear span of forests of rooted trees. If we take its dual, and identify the dual of each forest with the rooted tree having all trees of the forest grafted at its root vertex, the coproduct of Connes and Kreimer yields a product on the linear span of rooted trees which is defined in terms of graftings as above. Thus, this algebra is naturally endowed with another product which is compatible with the original one.
\end{example}

\begin{example}
Similarly, consider the noncommutative Connes--Kreimer Hopf algebra $NCK$ of Foissy~\cite{Foi}, which is a free associative algebra on planar rooted trees, or, in other words, an algebra on the linear span of (ordered) forests of planar rooted trees. If we take its dual, and identify the dual of each forest with the planar rooted tree having all trees of the forest grafted at its root vertex, the coproduct of Foissy leads to another product on $\mathbb{Q}\RT(S)$ which is again defined in terms of graftings. It follows that $NCK$ has a natural structure of an algebra with two compatible products. Results of Foissy~\cite{Foi} on isomorphisms of Hopf algebras also produce compatible products on some other algebras on trees, for example, the Brouder--Frabetti Hopf algebra of renormalization in QED~\cite{BF}.
\end{example}

One can easily check that unlike the case of the Grossman--Larson product, the dual of the Foissy algebra is not a free algebra with two compatible products; for example, the $\star_2$-subalgebra of $NCK$ generated by elements of degree~$1$ (i.e., trees with one leaf) is commutative. We expect that this is in some sense the only obstruction to freeness. 
\begin{conjecture}
$\As^2$-subalgebra of $NCK$ generated by elements of degree~$1$ is a free algebra over the operad of two compatible associative products one of which is, in addition, commutative.
\end{conjecture}
The operad that shows up here does not seem to have many good properties. In particular, it is not Koszul, and not much is known about the growth of dimensions of its components.

\subsection{General compatible structures}
It is natural to ask which of our results have analogues for other operads of compatible structures (for the formal definition of compatible $\calO$-structures for any operad~$\calO$, see \cite{Stro}). For the operad of compatible Lie brackets, we can prove an exact analogue of Theorem \ref{freearefree}: free algebras over that operad are free as Lie algebras. The proof is similar to the proof for compatible associative structures and also makes use of an appropriate monomial basis. We expect that actually both of these statements are particular cases of a very general theorem. 

The following conjecture consists of two parts. The first part generalises the main theorem of \cite{Stro}, while the second one suggests that our theorem also holds in that generality.

\begin{conjecture}
\begin{enumerate}
 \item Let $\calO$ be a Koszul operad. Then the operad $\calO^2$ of two compatible $\calO$-structures is also Koszul.
 \item If the operad $\calO$ is Koszul, then free $\calO^2$-algebras are free as $\calO$-algebras.
\end{enumerate}
\end{conjecture}


\begin{thebibliography}{XXXX}
\bibitem{BF} Christian Brouder and Alexandra Frabetti, QED Hopf algebras on planar binary trees,  J. Algebra  267  (2003),  no. 1, 298--322. 
\bibitem{Chap} Frederic Chapoton, Free pre-Lie algebras are free as Lie algebras, \texttt{arXiv:0704.2153v1 [math.RA]}, Bulletin 
canadien de math\'ematiques, to appear.
\bibitem{CK} Alain Connes and Dirk Kreimer, Hopf algebras, renormalization and noncommutative geometry, Comm.~Math.~Phys., 199 (1998), 203--242. 
\bibitem{DK} Vladimir Dotsenko and Anton Khoroshkin, Character formulas for the operad of two 
compatible brackets and for the bi-Hamiltonian operad, Funct.~Anal.~and~Appl., 2007, 41:1, 1--17.
\bibitem{Foi} Lo\"\i{}c Foissy, Les alg\`ebres de Hopf des arbres enracin\'es d\'ecor\'es. I, II. (French) [Hopf algebras of decorated rooted trees. I, II]  Bull. Sci. Math.  126  (2002),  no. 3, 193--239;  Bull. Sci. Math.  126  (2002),  no. 4, 249--288. 
\bibitem{GK} Victor Ginzburg and Mikhail Kapranov, Koszul duality for operads, Duke~Math.~J., 76:1 
(1994), 203--272. 
\bibitem{GL} Robert Grossman and Richard Larson, Hopf-algebraic structures of families of trees, 
J.~Algebra, 26 (1989), 184--210.
\bibitem{Lo} Jean--Louis Loday, Generalized bialgebras and triples of operads, \texttt{arXiv:math/0611885v2 [math.QA]}.
\bibitem{Lo_pc} Jean--Louis Loday, private communication.
\bibitem{Mac} Ian Macdonald, Symmetric functions and Hall polynomials, Oxford University Press, 
1995.
\bibitem{MSS} Martin Markl, Steve Shnider, Jim Stasheff, Operads in Algebra, Topology and Physics, 
Mathematical Surveys and Monographs, vol.\,96, AMS, Providence, Rhode Island, 2002.
\bibitem{Mo} Ieke Moerdijk, On the Connes--Kreimer construction of Hopf Algebras, Homotopy methods in algebraic topology (Boulder, CO, 1999),  311--321, Contemp. Math., 271, Amer. Math. Soc., Providence, RI, 2001.
\bibitem{OS} Alexander Odesskii and Vladimir Sokolov, Algebraic structures connected with pairs of 
compatible associative algebras, Int.~Math.~Res.~Notices (2006) Vol. 2006: article ID 43734.
\bibitem{Stan2} Richard Stanley, Enumerative combinatorics, Cambridge University Press, 1999.
\bibitem{Stro} Henrik Strohmayer, Operads of compatible structures and weighted partitions, 
\texttt{arXiv:0706.2196v2 [math.AT]}, J.~Pure~Appl.~Algebra, to appear.
\end{thebibliography}
\end{document}